\documentclass[12pt, twoside, leqno]{article}

\usepackage{amsmath,amsthm}
\usepackage{amssymb}
\usepackage{tikz}

\usepackage{enumerate}

\usepackage{graphicx}
\usepackage[english]{babel}
\usepackage{amsthm}
\usepackage[utf8]{inputenc}
\usepackage{mathrsfs}
\usepackage{bbm}
\usepackage{bm}
\usepackage{color}
\usepackage{hyperref}

\usepackage{nicefrac}

\theoremstyle{definition}

\numberwithin{equation}{section}

\newtheorem{tw}{Theorem}[section]
\newtheorem{lm}[tw]{Lemma}

\newtheorem{pr}[tw]{Proposition}

\theoremstyle{definition}
\newtheorem{df}{Definition}[section]
\newtheorem{uw}[tw]{Remark}

\newcommand{\R}{\mathbb{R}}
\newcommand{\Z}{\mathbb{Z}}
\newcommand{\N}{\mathbb{N}}
\newcommand{\T}{\mathbb{T}}

\newcommand{\Q}{\mathbb{Q}}

\newcommand{\cP}{\mathcal{P}}
\newcommand{\cT}{\mathcal{T}}

\newcommand{\bea}{\begin{eqnarray}}
  \newcommand{\eea}{\end{eqnarray}}
  \newcommand{\beab}{\begin{eqnarray*}}
  \newcommand{\eeab}{\end{eqnarray*}}
\renewcommand{\a}{\alpha}
  \newcommand{\be}{\begin{equation}}
  \newcommand{\ee}{\end{equation}}
  
\providecommand{\noopsort}[1]{}

\begin{document}


\baselineskip=17pt


\title{Product of two Kochergin flows with different exponents is not standard}

\author{Adam Kanigowski\\
Department of Mathematics\\
The Pennsylvania State University\\
University Park, PA 16802,USA\\
E-mail:adkanigowski@gmail.com
\and
Daren Wei\footnote{D. W. was partially supported by the NSF grant
DMS-16-02409}\\
Department of Mathematics\\
The Pennsylvania State University\\
University Park, PA 16802, USA\\
E-mail:duw170@psu.edu}

\date{}

\maketitle


\renewcommand{\thefootnote}{}



\renewcommand{\thefootnote}{\arabic{footnote}}
\setcounter{footnote}{0}


\begin{abstract}
We study the standard(zero entropy loosely Bernoulli or loosely Kronecker) property for products of Kochergin smooth flows on $\mathbb{T}^2$ with one singularity. These flows can be represented as special flows over irrational rotations of the circle and under roof functions which are smooth on $\mathbb{T}^2\setminus \{0\}$ with a singularity at $0$. We show that there exists a full measure set $\mathscr{D}\subset\mathbb{T}$ such that the product system of two Kochergin flows with different power of singularities and rotations from $\mathscr{D}$ is not standard.
\end{abstract}

\section{Introduction}
Standardness\footnote{The property of being standard  is often called loosely Bernoulli (with zero entropy) or loosely Kronecker. In this paper we will use the name standardness.} is a concept introduced by A. Katok\cite{Katok1} and J. Feldman \cite{Feldman}, \cite{Katok2}. In \cite{Feldman} first example of non-standard transformations were shown to exist by cutting and stacking method. First non-standard smooth examples on smooth manifolds (preserving a smooth measure) were constructed in \cite{Katok2}. Some other non-standard examples were constructed in \cite{OrnRudWei}. M. Ratner, \cite{Ratner1}, gave a natural (algebraic) example of a non-standard system. In \cite{Ratner1} it is shown that the cartesian product of the horocycle flow is not standard (the horocycle flow itself being standard, \cite{Rat3}). This gives natural examples in the smooth category in dimension $6$.

The aim of this paper is to study standardness for products from a different natural class of smooth flows. This are so called Kochergin flows, \cite{Koch}, which are smooth flows on $\T^2$ with one (degenerated) fixed point. One constructs them by slowing down the orbits of a linear flow on $\T^2$ (around the fixed point). Kochergin flows share many dynamical properties with the horocycle flows, i.e. they are standard (hence have $0$ entropy), mixing \cite{Koch}, some of them are mixing of all orders \cite{KanigowshiFayad}, have Lebesgue spectrum \cite{FFK} and the orbit growth is polynomial \cite{Kanigowski}. In this paper we show that the product of two Kochergin flows with one singularity and different exponents is not standard, which is one more dynamic feature showing similarity to the horocycle flows. To state our results more precisely we will pass to special representations of Kochergin flows with one singularity. They can be represented as special flows over irrational rotations and roof functions which are $C^2(\T\setminus \{0\})$ and the asymptotic behaviour around $0$ is given by $x^{\gamma}$,$-1<\gamma<0$ (see Section \ref{fuc} for a precise description).It follows that every Kochergin flow $(T_t)_{t\in \R}$ is given by a pair $(\alpha,\gamma)\in \R\setminus\Q\times(-1,0)$ and hence we denote $(T_t)_{t\in\R}=\cT^{\alpha,\gamma}$. With this notation we can state our main theorem.

\begin{tw}\label{MainTheorem}
There exists a full measure set $\mathscr{D}\subset \T$ such that for  every $\alpha_1,\alpha_2\in\mathscr{D}$ and every $\gamma_1,\gamma_2\in (-1,0)$, $\gamma_1\neq \gamma_2$, $\cT^{\alpha_1,\gamma_1}\times\cT^{\alpha_2,\gamma_2}$ is not standard.
\end{tw}

Let us emphasize that our proof in the current form rely on the fact $\gamma_1\neq \gamma_2$ and therefore we don't have a complete analogue of Ratner's result (see also Remark \ref{2.4}).
As a consequence of Theorem \ref{MainTheorem} we get natural examples of smooth non-standard systems in dimension $4$. Actually (see Remark \ref{noni}) our result gives uncountably many non-isomorphic, non-standard systems on $\T^4$.



\paragraph{Plan of the paper.}
In Section \ref{def} we give definitions of standardness, special flows, Kochergin flows and introduce some notations. Finally we give several lemmas based on Denjoy-Koksma inequality. In Section \ref{proofthmmain}, we state Proposition \ref{PropClose}, which describe a relation between $\bar{f}-$metric and the metric on the flow space. In Section \ref{proofpropropclose}, we conduct the proof of  Proposition \ref{PropClose} by  dividing the matching sequence (the sets $A_j^{R,m}$) along exponential scale (see Proposition \ref{PropEst}). In Section \ref{proofproest}, we prove Proposition \ref{PropEst} by Lemma \ref{EstLemma}, which is the most technical part and important part of our paper. In Section \ref{proofestlemma}, we finish the proof of Lemma \ref{EstLemma}.

\section{Basic definitions}\label{def}
\subsection{Standard}
Following Feldman we recall the definition of $\bar{f}$ metric. For two finite words (over a finite alphabet) $A=a_1...a_k$ and $B=b_1...b_k$ any pair of strictly increasing sequences $(i_s,j_s)_{s=0}^r$ such that $a_{i_s}=b_{j_s}$ for $s=0,1,...r$ is called a {\em matching} between $A,B$.
Then
\begin{displaymath}
\bar{f}_k(A,B)=1-\frac{r}{k},
\end{displaymath}

where $r$ is the maximal cardinality over all matchings between $A$ and $B$.
For zero entropy measure-preserving $T:(X,\mathscr{B},\mu)$, a finite partition $\cP=(P_0,P_1,...,P_r)$ of $X$ and a number $k\geq 0$ we denote $\cP_0^{k}(x):=x_0x_1...x_k$, where
$x_i\in\{0,1,...,r\}$ is such that $T^i(x)\in P_{x_i}$ for $i=0,1,...k$.
\begin{df}[\cite{Feldman}]\label{def:LB}
The process $(T,\cP)$ is said to be {\em standard (zero entropy loosely Bernoulli)} if for every $\epsilon>0$ there exists $N_{\epsilon}\in \mathbb{N}$ and a set $A_{\epsilon}\in \mathbb{B}$, $\mu(A_{\epsilon})>1-\epsilon$ such that for every $x,y\in A_{\epsilon}$ and $N\geq N_{\epsilon}$
\begin{displaymath}
\bar{f}_N(\cP_0^N(x),\cP_0^N(y))<\epsilon.
\end{displaymath}

The automorphism $T$ is {\em standard} if for every finite measurable partition $\cP$, the process $(T,\cP)$ is standard.
\end{df}

\begin{uw}
In order to simplify the notation, we denote \\ $\bar{f}_N^{\cP}(x,y)=
\bar{f}_N(\cP_0^N(x),\cP_0^N(y))$. If the parition is fixed, we will simply write $\bar{f}_N(x,y)$.
\end{uw}
We  have the following definition of standard for flows:
\begin{df}\label{lbflow} A {\em flow} $(T_t)_{t\in\R}$ is standard if it has a section on which the first return transformation is standard.
\end{df}
The following result proved in \cite{OrnRudWei} (for the $0$ entropy case) will be used in the paper (see also \cite{Tho}):
\begin{tw}\label{orw}
A flow $(T_t)_{t\in\R}$ is standard if and only if the time one automorphism $T_1$ is standard\footnote{This result was pointed to us by J-P. Thouvenot.}.
\end{tw}
We will use the above theorem for the proof of Theorem \ref{MainTheorem} by considering the time one map of the product flow. We will also use the following definition:

\begin{df}\label{epsgood}
Fix  a partition $\cP$, $N\in \N$ and $x,y\in X$. We say that a matching $(i_r,j_r)_{r=0}^{R(N)}$ of $x,y$ is {\em $(\cP,\epsilon)$-good} if
$R(N)\geq (1-\epsilon)N$.

\end{df}


\subsection{Special flows}
Let $\lambda$ be the Lebesgue measure on $\R$ (we will also denote by $\lambda$ the Lebesgue measure on $\T$, it will be clear from the context which space are we dealing with). Let $T:(X,\mathscr{B},\mu)\to(X,\mathscr{B},\mu)$ be an automorphism and $f\in L^1(X,\mathscr{B},\mu)$, $f>0$. The special flow $(T_t^f)_{t\in\mathbb{R}}$ acts on the space $(X^f,\mathscr{B}^f,\mu^f)$, where $X^f:=\{(x,s):x\in X,0\leq s<f(x)\}$, $\mathscr{B}^f:=\mathscr{B}\otimes\mathscr{B}(\mathbb{R})$ and $\mu^f=\mu\times\lambda$. Under the action of the flow $(T_t^f)_{t\in\R}$ each point in $X^f$ moves vertically with unit speed and we identify the point $(x,f(x))$ with $(Tx,0)$. More precisely, if $x=(x_h,x_v)\in X^f$ ($h$ and $v$ stand for horizontal and vertical coordinate) then
\begin{equation}\label{specfl}
  T_t^f(x_h,x_v)=(T^{N(x,t)}x_h,x_v+t-f^{(N(x,t))}(x_h)),
\end{equation}
$N(x,t)\in\mathbb{Z}$ is unique such that
\begin{displaymath}
 f^{(N(x,t))}(x_h)\leq x_v+t<f^{(N(x,t)+1)}(x_h)
\end{displaymath}
and
$$
f^{(n)}(x_h)=\left\{\begin{array}{ccc}
f(x_h)+\ldots+f(T^{n-1}x_h) &\mbox{if} & n>0\\
0&\mbox{if}& n=0\\
-(f(T^nx_h)+\ldots+f(T^{-1}x_h))&\mbox{if} &n<0.\end{array}\right.$$
Recall that if $X$ is a metric space with metric $d$, then so is $X^f$ with the product metric which we denote by $d^f$.
\subsection{Flows under consideration}\label{fuc}
Flows that we consider have the following special representation:
\begin{itemize}
    \item $T=R_\alpha:\T \to\T$,  $R_\alpha x=x+\alpha \mod 1$;
	 \item $f$ is a $C^2(\T\setminus\{0\})$ function which satisfies for some $-1<\gamma<0$ and $A_1,B_1>0$
\begin{equation}\label{Kor1}
\lim_{x\to0^+}\frac{f(x)}{x^{\gamma}}=A_1\text{ and }\lim_{x\to0^-}\frac{f(x)}{(1-x)^{\gamma}}=B_1,
\end{equation}
\begin{equation}\label{Kor2}
\lim_{x\to0^+}\frac{f'(x)}{x^{-1+\gamma}}=\gamma A_1\text{ and }\lim_{x\to0^-}\frac{f'(x)}{(1-x)^{-1+\gamma}}=-\gamma B_1,
\end{equation}
\begin{equation}\label{Kor3}
\begin{aligned}
&\lim_{x\to0^+}\frac{f''(x)}{x^{-2+\gamma}}=\gamma(\gamma-1)A_1\text{ and }\lim_{x\to0^-}\frac{f''(x)}{(1-x)^{-2+\gamma}}=\gamma(\gamma-1)B_1.
\end{aligned}
\end{equation}
\end{itemize}
We call such flows {\em Kochergin flows} and denote them by $\cT^{\a,\gamma}$. In what follows let $(q_{\alpha,n})_{n\geq 1}$ denote the sequence of denominators of $\alpha\in \R\setminus\Q$ Let us introduce the following set:
\begin{equation}\label{setd}
\mathscr{D}:=\{\a\in \R\setminus\Q: q_{\alpha,n+1}<C(\a)q_{\alpha,n}\log q_{\alpha,n}(\log n)^2\},
\end{equation}

It follows from Khinchin theorem, \cite{Kinchin}, that $\lambda(\mathscr{D})=1$. By Theorem \ref{orw}, Theorem \ref{MainTheorem} follows by the following :
\begin{tw}\label{mainth} Let $\gamma_1,\gamma_2\in(-1,0)$, $\gamma_1\neq \gamma_2$ and let $\alpha_1,\alpha_2\in \mathscr{D}$. Then the time one map of $\cT^{\alpha_1,\gamma_1}\times \cT^{\alpha_2,\gamma_2}$ is not standard.
\end{tw}
Let us make two remarks.
\begin{uw}\label{2.4} The set $\mathscr{D}$ in Theorem \ref{mainth} is not optimal. The authors think for $\xi>0$ (sufficiently small) Theorem \ref{mainth} holds for
$$
\mathscr{D'}:=\{\a\in \R\setminus\Q: q_{\alpha,n+1}<C(\a)q_{\alpha,n}^{1+\xi}\}.
$$
This would however need some more exact estimates of Denjoy-Koksma type which would complicate (already technical) proofs. Therefore we restrict to a smaller set of irrationals. It is also interesting what happens if we consider the case $\gamma_1=\gamma_2$, in particular a (cartesian) square of a given flow. In this case the proof is more complicated and would require some additional work, but the authors think it is also possible to have analogous results with the same exponent. Let us just emphasize that in the current form the proofs rely quite strongly on the fact that the exponents are different.
\end{uw}

\begin{uw}\label{noni} Theorem \ref{mainth} also gives uncountably many non-isomorphic non standard smooth flows in dimension $4$. The non-isomorphism follows provided that $\gamma_1+\gamma_2\neq \gamma_1'+\gamma_2'$ since in this case one can use slow entropy type results for these flows obtained by the first author in \cite{Kanigowski} to distinguish corresponding flows. It seems that by a more careful analysis one can provide an asymptotics for the number of $\bar{f}$ balls for the product system. This way it should be possible to have uncountably many smooth, natural, non Kakutani equivalent flows on $\T^4$. This should be compared with Benhenda\cite{Benhenda}.
\end{uw}



\subsection{Notation}\label{not}
For $i=1,2$ let $\cT^{f_i}=(T^{f_i}_t)_{t\in\R}$ denote the flow $\cT^{\alpha_i,\gamma_i}$ acting on $\T^{f_i}$. Recall that metric $d_i$ on $\T^{f_i}$ is of the form $d_i(x_1,x_2)=d_H(x_1,x_2)+d_V(x_1,x_2)$, where $d_H$ and $d_V$ are usual distances on $\T$ and $\R$ respectively\footnote{$H$ and $V$ stand for respectively horizontal and vertical.}. For $A\subset \T$ let $A^f:=\{x=(x_h,x_v)\in\T^f:x_h\in A\}$. A partition $\cP$ of $\T^{f_1}\times \T^{f_2}$ is fixed in this section.


\begin{df}\label{eano}
For 
$(x,y), (x',y')\in \T^{f_1}\times \T^{f_2}$, $N\in \N$  and a matching
$(i_s,j_s)_{s=0}^{R(N)}$ 
let
\begin{equation}
\begin{aligned}
&(x_w,y_w)=(T^{f_1}\times T^{f_2})_{i_w}(x,y),\\
&(x'_w,y'_w)=(T^{f_1}\times T^{f_2})_{j_w}(x',y'),
\end{aligned}
\end{equation}
for $w\in \{0,1,...,R(N)\}$. Moreover denote
\begin{equation}\label{hor}
L_H(r):=\max\{d_H(x_r,x'_r),d_H(y_r,y'_r)\},
\end{equation}
\begin{equation}\label{total}
L(r):=\max\{d^{f_1}(x_r,x'_r),d^{f_2}(y_r,y'_r)\}.
\end{equation}
\end{df}


Notice that $L_H(r)$ and $L(r)$ depend on $N, x,y,x',y'$ and a matching $(i_s,j_s)_{s=0}^{R(N)}$. In the proofs all of the above will be fixed and therefore we will use the short notation as in \eqref{hor} and \eqref{total}.

\begin{df}[\textbf{Matching balls}]
Fix $N\in\N$, $(x,y),(x',y')\in \mathbb{T}^{f_1}\times\mathbb{T}^{f_2}$ and a matching $(i_s,j_s)_{s=0}^{R(N)}$. For $0\leq r,w\leq R(N)$ and $U>0$, we define
\begin{equation}
(i_r,j_r)\in B((i_w,j_w),U)\Leftrightarrow\{i_r\in[i_w-U,i_w+U],\ \ j_r\in[j_w-U,j_w+U]\}.
\end{equation}
\end{df}

Recall that for a fixed $N$, a matching is a sequence $(i_s,j_s)_{s=0}^{R(N)}$. To simplify notation, we will write $(i_s,j_s)_{s=0}^N$ with the understanding that for $k>R(N)$, $(i_k,j_k)=(N+1,N+1)$.

\subsection{Choice of the partition}\label{part}
By Definition \ref{def:LB} and Theorem \ref{orw} to prove Theorem \ref{mainth} it is enough to find one partition of the form $\cP_1\times \cP_2$ of $\T^{f_1}\times \T^{f_2}$ for which the automorphism $T_1^{f_1}\times T^{f_2}_1$ is not standard. There is a natural sequence of partitions of $\T^{f_i}$, $i=1,2$. One has to cut off the cusp at some height and divide the compact part into rectangles of small diameters. More precisely, for $m\in \N$ and $i=1,2$ let $P^i_m$ be the partition obtained by dividing the set $K^i_m:= \{x\in \T^{f_i} : f_i(x_h) < 2^m\}$ into (finitely many) sets (atoms) of diameter between $\frac{1}{m}$ and $\frac{2}{m}$ (with a $C^1$ boundary) and taking $\T^{f_i}\setminus K^i_m$ to be one atom. We will show that for sufficiently large $m$ the automorphism is not standard for $\cP_m:=P^1_m\times P^2_m$.






\subsection{Denjoy-Koksma Estimates}
In this section we will state some lemmas describing the behaviour of ergodic sums over an irrational rotation $\alpha$ for functions with power singularities. The proofs follow mainly from the Denjoy-Koksma inequality and since the methods are classical for this type of functions (see e.g. \cite{FFK}) we will give the proofs in the appendix. We will consider a function $f$ satisfying \eqref{Kor1}, \eqref{Kor2}, \eqref{Kor3} with some $-1<\gamma<0$ and $A_1=B_1=1$. For simplicity we will also assume that $\int_{\mathbb{T}}fd\lambda=1$.



 Recall that $\alpha\in \mathscr{D}$ (see \eqref{setd}).

For $z\in\mathbb{T}$ denote
\begin{displaymath}
  z_{\min}^M=\min_{j\in[0,M)}\|z+j\alpha-0\|.
\end{displaymath}

The following lemma can be found \cite{FFK}, Lemma $3.1$.
\begin{lm}\label{koks}
For every $z\in\mathbb{T}$ and every $M\in\mathbb{Z}$, $|M|\in[q_{\alpha,s},q_{\alpha,s+1}]$ we have
\begin{equation}\label{DK1}
  f(z_{\min}^M)+\frac{1}{3}q_{\alpha,s}\leq f^{(M)}(z)\leq f(z_{\min}^M)+3q_{\alpha,s+1},
\end{equation}
\begin{equation}\label{DK2}
f'(z_{\min}^M)-8|\gamma|
q_{\alpha,s}^{1+|\gamma|}\leq |f'^{(M)}(z)|\leq f'(z_{\min}^M)+
8|\gamma|q_{\alpha,s+1}^{1+|\gamma|},
\end{equation}
and
\begin{equation}\label{DK3}
  f''(z_{\min}^M)\leq f''^{(M)}(z)\leq f''(z_{\min}^M)+
8|\gamma(\gamma-1)|q_{\alpha,s+1}^{2+|\gamma|}.
\end{equation}
\end{lm}

The following sets will be used in the proof of Theorem \ref{mainth}.

Points in the sets below approach the singularity in a controled way and we have for them a nice upper bounds for ergodic sums of the first and second derivative. For simplicity of notation from now on we denote the sequence of denominators of $\alpha_1$ and $\alpha_2$ respectively by $(q_n)_{n\geq 1}$ and $(q'_n)_{n\geq 1}$.
\begin{equation}\label{sin}
\begin{aligned}
&S_n^1=\left\{x\in\mathbb{T}^{f_1}:\bigcup_{t=-q_n\log q_n}^{q_n\log q_n}T_t^{f_1}x\nsubseteq\left[-\frac{1}{q_n\log^3q_n},\frac{1}{q_n\log^3q_n}\right]^{f_1}\right\},\\
&S_n^2=\left\{x\in\mathbb{T}^{f_2}:\bigcup_{t=-q'_n\log q'_n}^{q'_n\log q'_n}T_t^{f_2}x\nsubseteq\left[-\frac{1}{q'_n\log^3q'_n},\frac{1}{q'_n\log^3q'_n}\right]^{f_2}\right\}.
\end{aligned}
\end{equation}

We have
\begin{equation}\label{smlmes}
\mu(S_n^1)\geq 1-\frac{c(f_1)}{\log^2q_n}\text{ and }\mu(S_n^2)\geq 1-\frac{c(f_2)}{\log^2q'_n},
\end{equation}
for some constant $c(f_i)>0$, $i=1,2$. For $n_1\in \N$, we define
\begin{equation}\label{si}
S^i(n_1):=\bigcap_{n\geq n_1}S_n^i,\;\;\; i=1,2.
\end{equation}
Notice that by \eqref{smlmes} for every $\delta>0$ there exists $n_1=n_1(\delta)\in\mathbb{N}$ such that $\mu(S^i(n_1))>1-\delta^3$ for $i=1,2$. We have the following lemma which is proved in the Appendix (recall that for $z\in \T^{f_i}$, $z_h$ denotes the first coordinate of $z$):

\begin{lm}\label{new.l} Fix $i=1,2$. There exists a constant $d_i=d_i(f_i,\alpha_i)>0$ such that for $z,z'\in S^i(n_1)$ satisfying $d_H(z,z')\leq d_i$ and $z_h<z_h'$ we have for every
$0<w<|\frac{d_H(z,z')^{-1}}{\log^{12} d_H(z,z')}|$
$$
0\notin [z_h+w\alpha_i,z_h'+w\alpha_i].
$$
\end{lm}

Let $P_i=100|\gamma_i|^{-1}$, $i=1,2$ and define
\begin{equation}
\begin{aligned}
&W_{t}^1=\left\{x\in \mathbb{T}^{f_1}:|f_1'^{(N(x,t))}(x_h)|\geq \frac{|N(x,t)|^{1+|\gamma_1|}}{\log^{P_1}|N(x,t)|}\right\}, \\
&W_{t}^2=\left\{y\in \mathbb{T}^{f_2}:|f_2'^{(M(y,t))}(y_h)|\geq \frac{|M(y,t)|^{1+|\gamma_2|}}{\log^{P_2}|M(y,t)|}\right\}.\\
\end{aligned}
\end{equation}
where $N(x,t)$ and $M(y,t)$ are defined in \eqref{specfl} for respectively $\cT^{f_1}$ and $\cT^{f_2}$.
We will use the following proposition (see Proposition 6.8 in \cite{Kanigowski}).
\begin{pr}\label{SetPro} Fix $\delta>0$.
There exists sets $W^1\subset\mathbb{T}^{f_1}$, $W^2\subset\mathbb{T}^{f_2}$, $\mu(W^i)\geq1-\delta^{10}$ for $i=1,2$ and $n_2=n_2(\delta)\in\mathbb{N}$ such that for every $z\in W^i$ and $T\geq n_2$ we have

\begin{equation}\label{wti}
\lambda(\{t\in[-T,T]:z\in W_{t}^i\})\geq 2T(1-\log^{-3}T),
\end{equation}

for $i=1,2$.
\end{pr}

The following lemma describes the behaviour of ergodic sums. The proof is carried over to the Appendix.
\begin{lm}\label{ConDer}For every $\epsilon_2>0$ there  exists $n'\in \N$ and $\delta_0>0 $such that for every $\delta_0>\delta>0$, $x\in S^1(n_1)\cap W^{1}$,  $y\in S^2(n_1)\cap W^2$ and $T\geq n'$ there exists a set $G_T\subset [0,T]$, $\lambda(G_T)\geq T(1-4\log^{-3}T)$ such that for every $t\in G_T$, we have
\begin{equation}\label{MainRelation}
\begin{aligned}
&t^{1+|\gamma_1|-\epsilon_2}\leq
|f_1'^{(N(x,t))}(\theta_h)| \leq t^{1+|\gamma_1|+\epsilon_2},\text{ for every } d_H(\theta_h,x_h)\leq (Tlog^{2P_1}T)^{-1},\\
&t^{1+|\gamma_2|-\epsilon_2}\leq|f_2'^{(M(y,t))}(\xi_h)|\leq t^{1+|\gamma_2|+\epsilon_2},\text{ for every } d_H(\xi_h,y_h)\leq (Tlog^{2P_2}T)^{-1}.
\end{aligned}
\end{equation}
\end{lm}
In order to simplify the proof, we suppose that $|\gamma_1|>|\gamma_2|$.
\begin{df}\label{eps0} Let $\epsilon_0>0$ and $\epsilon_2>0$ be such that there exists $R_0=R_0(\gamma_1,\gamma_2)$ such that for every $R\geq R_0$, we have
\begin{equation}\label{rweps0}
\frac{R^{(\frac{1}{1+|\gamma_2|}+\frac{1}{2}\epsilon_0)(1+|\gamma_2|
-\epsilon_2)-1}-\frac{1}{2}}
{R^{(\frac{1}{1+|\gamma_2|}+\epsilon_0)(1+|\gamma_1|+\epsilon_2)}}>
\frac{R^{(1-\epsilon_0)(1+|\gamma_2|+\epsilon_2)-1}+\frac{1}{2}}
{R^{(1-2\epsilon_0)(1+|\gamma_1|-\epsilon_2)}},
\end{equation}
and
$(1-2\epsilon_0)(1+|\gamma_1|-\epsilon_2)>1+|\gamma_2|+\epsilon_2$.
\end{df}
One can show that $\epsilon_0=
\frac{2\epsilon_2}{|\gamma_2|(1+|\gamma_2|)}$ and $\epsilon_2>0$ small enough (smallness depending only on $\gamma_1$ and $\gamma_2$!) satisfies Definition \ref{eps0}. From now on parameters $\epsilon_0,\epsilon_2>0$ are fixed throughout the paper.

\section{Proof of Theorem \ref{mainth}}\label{proofthmmain}
We will use the following proposition to prove Theorem \ref{mainth}. Assume $\gamma_i,\alpha_i$, $i=1,2$ are as in Theorem \ref{mainth}. Recall that the sequence $\cP_m$ is defined in Section \ref{part} and we use notation from Section \ref{not}.



Let $\epsilon_0>0$ be as in Definition \ref{eps0}.
\begin{pr}\label{PropClose}
For every $\delta>0$ there exists a set $A=A_{\delta}\subset\mathbb{T}^{f_1}\times\mathbb{T}^{f_2}$, $\mu^{f_1}\times \mu^{f_2}(A)>1-\delta$ and $m_{\delta},R_{\delta}\in\mathbb{N}$ such that for every $(x,y),(x',y')\in A$, $m\geq m_{\delta}$, $R\geq R_{\delta}$ and every $(\cP_m,\frac{1}{100})$-good matching $(i_s,j_s)_{s=0}^R$ of $(x,y)$ and $(x',y')$ there exists $(i_{r_0},j_{r_0})\in [0,R]^2\cap \Z$ such that
\begin{equation}
L_H(r_0)\leq R^{-\frac{1}{1-\epsilon_0}}.
\end{equation}
\end{pr}

Before we prove Proposition \ref{PropClose}, let us show how it implies Theorem \ref{mainth}.
\begin{proof}[Proof of Theorem \ref{mainth}]
Fix $\delta=1/10$ and let $m=m_\delta$, $R\geq R_\delta$ and $A=A_\delta$ be as in Proposition \ref{PropClose}. To finish the proof we will show that for every $x,y\in A$, we have
$$
\mu^{f_1}\times\mu^{f_2}(A\cap B_{\bar{f}_R^{P_m}}((x,y),1/100))
< 1/2.
$$
 Let $C_1:=(\min_\T f_1)^{-1}$. The above will follow by showing that for every $(x,y),(x',y')\in A\times \T^{f_2}$ satisfying $\bar{f}_R^{P_m}((x,y),(x',y'))<\frac{1}{100}$, we have
\begin{equation}\label{imp}
x'_h\in\bigcup^{2C_1R}_{i=-2C_1R}\left[x_h+
i\alpha_1-R^{-\frac{1}{1-\epsilon_0}},
x_h+i\alpha_1+R^{-\frac{1}{1-\epsilon_0}}\right]:=S(x).
\end{equation}
(recall  that $x_h,x'_h\in\T$ denote first coordinates of $x,x'$). Indeed, notice that $\lambda(S(x))\leq 4C_1RR^{-\frac{1}{1-\epsilon_0}}\ll 1/2$ and by \eqref{imp}, we have
$$
A\cap B_{\bar{f}_R^{P_m}}((x,y),1/100)\subset \{x': x'_h\in S(x)\}\times \T^{f_2}.
$$
So it is enough to show \eqref{imp}. By definition, there exists a $(\cP_m,\frac{1}{100})$-good matching $(i_s,j_s)_{s=0}^R$ of $(x,y)$ and $(x',y')$. By Proposition \ref{PropClose} we know there exist $(i_{r_0},j_{r_0})\in[0,R]^2\cap \Z$ 
such that $L_H(r_0)\leq R^{-\frac{1}{1-\epsilon_0}}$. In particular it follows that
\begin{equation}\label{hora}d_H(x_{r_0},x'_{r_0})\leq R^{-\frac{1}{1-\epsilon_0}}.
\end{equation}
Let $r_1,r_2\geq 0$ be such that the $x_h+r_1\alpha_1$ and $x'_h+r_2\alpha_1$ are the first coordinates of $x_{r_0}=$ and $x'_{r_0}$. By the definition of special flow it follows that $C_1R\geq r_1,r_2$. Then \eqref{hora} becomes
\begin{equation}\label{localInequality}
\|x_h-x'_h-(r_2-r_1)\alpha_1\|\leq R^{-\frac{1}{1-\epsilon_0}}.
\end{equation}
This finishes the proof of \eqref{imp} since $|r_2-r_1|\leq r_1+r_2\leq 2RC_1$.

\end{proof}

\section{Proof of Proposition \ref{PropClose}}\label{proofpropropclose}

For $m>0$ and $R,j\in\mathbb{N}$ $(x,y)$ and $(x',y')$ we fix a $(\cP_m,\frac{1}{100})$-good matching $(i_s,j_s)_{s=0}^{R}$ of $(x,y)$ and $(x',y')$. Define (see Section \ref{not})
\begin{equation}
A_j^{R,m}((x,y),(x',y'))=\{r\in[0,R]:L(r)<2m^{-1} and \ \ 2^{-j-1}<L_H(r)\leq2^{-j}\}
\end{equation}

Let us explain the meaning of the sets $A_j^{R,m}$. In  Figure 1 we have a matching between $(x,y)$ and $(x',y')$. To each such arrow we prescribe a parameter which measures how large is the horizontal distance between the first coordinates on both endpoints of the arrow (see Figure 2). In Figure 2 different colors correspond to different $j$'s in sets $A^{R,m}_j$.


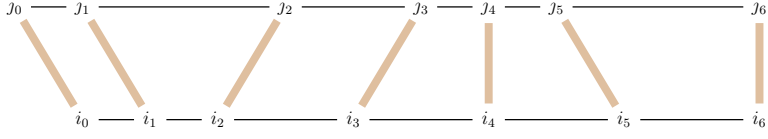
\begin{figure}[h]
 \centering
 \scalebox{0.6}
 {
 \begin{tikzpicture}[scale=5]
	 \tikzstyle{vertex}=[circle,minimum size=20pt,inner sep=0pt]
	 \tikzstyle{selected vertex} = [vertex, fill=red!24]
	 \tikzstyle{edge1} = [draw,line width=5pt,-,red!50]
     \tikzstyle{edge2} = [draw,line width=5pt,-,green!50]
     \tikzstyle{edge3} = [draw,line width=5pt,-,blue!50]
     \tikzstyle{edge4} = [draw,line width=5pt,-,brown!50]

	 \tikzstyle{edge} = [draw,thick,-,black]
	 \node[vertex] (v01) at (0.3,0) {$i_0$};
	 \node[vertex] (v02) at (0.6,0) {$i_1$};
	 \node[vertex] (v03) at (0.9,0) {$i_2$};
	 \node[vertex] (v04) at (1.5,0) {$i_3$};
	 \node[vertex] (v05) at (2.1,0) {$i_4$};
	 \node[vertex] (v06) at (2.7,0) {$i_5$};
	 \node[vertex] (v07) at (3.3,0) {$i_6$};
     \node[vertex] (v10) at (0,0.5) {$j_0$};
	 \node[vertex] (v11) at (0.3,0.5) {$j_1$};
	 \node[vertex] (v12) at (1.2,0.5) {$j_2$};
	 \node[vertex] (v13) at (1.8,0.5) {$j_3$};
	 \node[vertex] (v14) at (2.1,0.5) {$j_4$};
	 \node[vertex] (v15) at (2.4,0.5) {$j_5$};
	 \node[vertex] (v16) at (3.3,0.5) {$j_6$};
	
	 \draw[edge] (v01)--(v02)--(v03)--(v04)--(v05)--(v06)--(v07);
	 \draw[edge] (v10)--(v11)--(v12)--(v13)--(v14)--(v15)--(v16);
     \draw[edge4] (v10)--(v01);
     \draw[edge4] (v11)--(v02);
     \draw[edge4] (v12)--(v03);
     \draw[edge4] (v13)--(v04);
     \draw[edge4] (v14)--(v05);
     \draw[edge4] (v15)--(v06);
     \draw[edge4] (v16)--(v07);

 \end{tikzpicture}
 }
 \caption{Original Matching}
 \end{figure}


 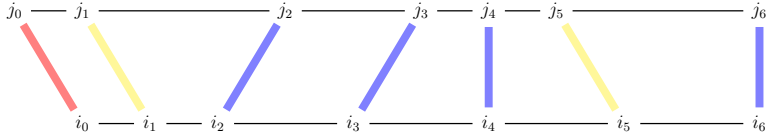
\begin{figure}[h]\label{fig2}
 \centering
 \scalebox{0.6}
 {
 \begin{tikzpicture}[scale=5]
	 \tikzstyle{vertex}=[circle,minimum size=20pt,inner sep=0pt]
	 \tikzstyle{selected vertex} = [vertex, fill=red!24]
	 \tikzstyle{edge1} = [draw,line width=5pt,-,red!50]
     \tikzstyle{edge2} = [draw,line width=5pt,-,yellow!50]
     \tikzstyle{edge3} = [draw,line width=5pt,-,blue!50]

	 \tikzstyle{edge} = [draw,thick,-,black]
	 	 \node[vertex] (v01) at (0.3,0) {$i_0$};
	 \node[vertex] (v02) at (0.6,0) {$i_1$};
	 \node[vertex] (v03) at (0.9,0) {$i_2$};
	 \node[vertex] (v04) at (1.5,0) {$i_3$};
	 \node[vertex] (v05) at (2.1,0) {$i_4$};
	 \node[vertex] (v06) at (2.7,0) {$i_5$};
	 \node[vertex] (v07) at (3.3,0) {$i_6$};
     \node[vertex] (v10) at (0,0.5) {$j_0$};
	 \node[vertex] (v11) at (0.3,0.5) {$j_1$};
	 \node[vertex] (v12) at (1.2,0.5) {$j_2$};
	 \node[vertex] (v13) at (1.8,0.5) {$j_3$};
	 \node[vertex] (v14) at (2.1,0.5) {$j_4$};
	 \node[vertex] (v15) at (2.4,0.5) {$j_5$};
	 \node[vertex] (v16) at (3.3,0.5) {$j_6$};
	
	 \draw[edge] (v01)--(v02)--(v03)--(v04)--(v05)--(v06)--(v07);
	 \draw[edge] (v10)--(v11)--(v12)--(v13)--(v14)--(v15)--(v16);
     \draw[edge1] (v10)--(v01);
     \draw[edge2] (v11)--(v02);
     \draw[edge3] (v12)--(v03);
     \draw[edge3] (v13)--(v04);
     \draw[edge3] (v14)--(v05);
     \draw[edge2] (v15)--(v06);
     \draw[edge3] (v16)--(v07);

 \end{tikzpicture}
 }
 \caption{Partition of Matching Arrows based on set $A_j^{R,m}$}
 \end{figure}

We will use the following proposition to show Proposition \ref{PropClose}.
\begin{pr}\label{PropEst} For every $\delta>0$ there exists $R_{\delta},m_{\delta}>0$ and a set $B=B_{\delta}\subset\mathbb{T}^{f_1}\times\mathbb{T}^{f_2}$, $\mu^{f_1}\times\mu^{f_2}(B)>1-\delta$, such that for every $m\geq m_{\delta}$, $R\geq R_{\delta}$ and every $(x,y),(x',y')\in B$ there exists $U_R(x,y),U_R(x',y')\subset[0,R]$ such that
\begin{itemize}
\item[\textbf{(a)}] $|U_R(x,y)|,|U_R(x',y')|\geq\frac{99}{100}R$,
\item[\textbf{(b)}] for every $(p,q)\in U_R(x,y)\times U_R(x',y')$, we have $(T^{f_1}\times T^{f_2})_p(x,y),(T^{f_1}\times T^{f_2})_q(x',y')\in K_m^{1}\times K_m^{2}$ (see Section \ref{part}),
\item[\textbf{(c)}] for every $j\in\mathbb{N}$ such that $2^j\leq R^{\frac{1}{1-\epsilon_0}}$, and any $(\cP_m,\frac{1}{100})$-good matching $(i_r,j_r)_{r=0}^R$ of $(x,y)$ and $(x',y')$, we have
    \begin{equation}
    \begin{aligned}
&\left|A^{R,m}_j((x,y),(x',y'))\cap\{r\in [0,R]:(i_r,j_r)\in U_R(x,y)\times U_R(x',y')\}\right|\\&\leq\frac{R}{j^{1.4}}.
    \end{aligned}
    \end{equation}
\end{itemize}
\end{pr}
We will give a proof of Proposition \ref{PropEst} in the next section. Let us first show how it implies Proposition \ref{PropClose}.

Before we give a complete proof let us give a sketch.
\paragraph{Sketch of the proof:} Fix a $(\cP_m,1/100)$-good matching $(i_s,j_s)_{s=0}^R$ of $(x,y)$ and $(x',y')$. It follows by the definition of $\cP_m$ that for most times  (see \textbf{(b)}) (except times where the orbits are close to the cusp) if the orbits are in one atom of $\cP_m$, then horizontal distances of $x,x'$ and $y,y'$ have to be small. Hence they have to belong to the set $\bigcup_{j}A_j^{R,m}$. Notice that the sets $(A_j^{R,m})_{j}$ are disjoint. Hence the cardinality of the matching is bounded above by the sum over the $(A_j^{R,m})_{j\in\N}$. Now by \textbf{(c)} it follows that this sum is small if the maximal horizontal distance is $\geq R^{-\frac{1}{1-\epsilon_0}}$ (the exponent 1.4 in \textbf{(c)} is just to make the series  $\sum_j j^{-1.4}$ summable). Hence if the matching occupies most of $[0,R]\cap \Z$, the sum has to be large and then there has to be a time where the horizontal distance is $\leq R^{-\frac{1}{1-\epsilon_0}}$.

\begin{proof}[Proof of Proposition \ref{PropClose}]
Fix $\delta>0$, $R\geq R_{\delta}$, $m\geq m_{\delta}$, $(x,y),(x',y')\in B$ and a  $(\cP_m,\frac{1}{100})$-good matching $(i_r,j_r)_{r=0}^R$ of $(x,y)$ and $(x',y')$. Notice that by \textbf{(b)}, the definition of $K_m^1$ and $K_m^2$ and the definition of $(A_j^{R,m})_{j\in \N}$, we have
$$
\{r\in\{0,...,R\}: (i_r,j_r)\in U_R(x,y)\times U_R(x',y')\}\subset \bigcup_{j\in \N} A_j^{R,m}.
$$


Hence, by $(\textbf{a})$, $(\textbf{b})$ and the definition of $A_j^{R,m}$, we have
\begin{equation}\label{ala}
\begin{aligned}
&\frac{1}{100}>\bar{f}_R^{P_m}((x,y),(x',y'))\\
&\geq 1-\frac{|(U_R(x,y))^c\cap (U_R(x',y'))^c\cap[0,R]|}{R}-\\
&\frac{1}{R}\sum_{j\geq0}|\{r\in [0,R]:(i_r,j_r)\in U_R(x,y)\times U_R(x',y')\}\cap A^{R,m}_j((x,y),(x',y'))|\\
&\geq\frac{9}{10}-\\&\frac{1}{R}\sum_{j\geq0}|\{r\in [0,R]:(i_r,j_r)\in U_R(x,y)\times U_R(x',y')\}\cap A^{R,m}_j((x,y),(x',y'))|
\end{aligned}
\end{equation}

Notice that by the definition of $A_j^{R,m}$, for $j\leq\frac{\log m}{2}$, we have
\begin{equation}\label{wpan}
\{r\in [0,R]:(i_r,j_r)\in U_R(x,y)\times U_R(x',y')\}\cap A^{R,m}_j((x,y),(x',y'))=\emptyset.
\end{equation}
Let $j_R$ be such that,
\begin{equation}\label{er}
2^{j_R}\leq R^{\frac{1}{1-\epsilon_0}}<2^{j_R+1}.
\end{equation}

As $\log m\gg 1$, by $(\textbf{c})$ and \eqref{wpan} we have
\begin{equation}
\begin{aligned}
\frac{1}{R}\sum_{j< j_R}|\{r\in [0,R]:(i_r,j_r)\in U_R(x,y)\times U_R(x',y')\}&\cap A^{R,m}_j((x,y),(x',y'))|\\
&\leq\frac{1}{R}\sum_{\frac{\log m}{2}\leq j\leq j_R-1}\frac{R}{j^{1.4}}\\
&\leq\frac{1}{1000}.
\end{aligned}
\end{equation}

Therefore and by \eqref{ala} there exists $j_1\geq j_R$ such that
\begin{equation}
\{r\in [0,R]:(i_r,j_r)\in U_R(x,y)\times U_R(x',y')\}\cap A^{R,m}_{j_1}((x,y),(x',y'))\neq\emptyset.
\end{equation}

By definition of $A^{R,m}_{j_1}$ and \eqref{er} it follows that there exists $(i_{r_0},j_{r_0})\in[0,R]^2$ such that
\begin{equation}
L_H(r_0)\leq R^{-\frac{1}{1-\epsilon_0}}.
\end{equation}

This finishes the proof of Proposition \ref{PropClose}.
\end{proof}

\section{Proof of Proposition \ref{PropEst}}\label{proofproest}
In this section we will state a lemma which implies Proposition \ref{PropEst}.


\begin{lm}\label{EstLemma} For every $\delta>0$ there exists $R_{\delta},m_{\delta}\in\mathbb{N}$ and $B=B_\delta\subset \T^{f_1}\times \T^{f_2}$, $\mu^{f_1}\times\mu^{f_2}(D)>1-\delta$
such that for every $R\geq R_{\delta}$, $m\geq m_{\delta}$, every $(x,y),(x',y')\in D$ there exists $U_R(x,y),U_R(x',y')\subset [0,R]$ such that \textbf{(a)},\textbf{(b)} hold and for every $(\cP_m,\frac{1}{100})$-good matching $(i_r,j_r)_{r=0}^{R}$ of $(x,y)$ and $(x',y')$,  and we have
\begin{itemize}

\item[$(1)$] For $(i_w,j_w)\in U_{R}(x,y)\times U_R(x',y')$ if $R_w^{-1}:=L_H(w)<2m^{-1}$, then for every $(i_r,j_r)\in B((i_w,j_w),\frac{R_{w}}{\log^5R_{w}})$ either
    \begin{equation}\label{isom}
    d_H(x_r,x'_r)=d_H(x_w,x'_w)\text{ and } d_H(y_r,y'_r)=d_H(y_w,y'_w) 
    \end{equation} or
    \begin{equation}\label{largehor}
    L_H(r)\geq 100 L_H(w).
    \end{equation}
\item[$(2)$] for every $(i_w,j_w)\in U_{R}(x,y)\times U_R(x',y')$ such that $L(w)<2m^{-1}$, we have at least one of the following inequalities
    \begin{equation}\label{sh}
    \begin{aligned}
      |\{r\in[-R_w^{\frac{1}{1+|\gamma_2|}+\epsilon_0},&
R_w^{\frac{1}{1+|\gamma_2|}+\epsilon_0}]:(i_r,j_r)\in B((i_w,j_w),R_w^{\frac{1}{1+|\gamma_2|}+\epsilon_0}),\\
      &\eqref{isom} \text{ holds }\ \ and\ \ L(r)<1/4\}|<\frac{R_w^{\frac{1}{1+|\gamma_2|}+\epsilon_0}}{\log^2R_w}
    \end{aligned}
    \end{equation}
    or
    \begin{equation}\label{lon}
    \begin{aligned}
      |\{r\in[-R_w^{1-\epsilon_0},R_w^{1-\epsilon_0}]:&(i_r,j_r)\in B((i_w,j_w),R_w^{1-\epsilon_0}),\\
      &\eqref{isom} \text{ holds }\ \ and\ \ L(r)<1/4\}|<\frac{R_w^{1-\epsilon_0}}{\log^2R_w}.
    \end{aligned}
    \end{equation}

\end{itemize}
\end{lm}

We will show how Lemma \ref{EstLemma} implies Proposition \ref{PropEst}. Before that we will sketch the main ideas.
\paragraph{Sketch of the proof.} The main difficulty is the proof of \textbf{(c)} in Proposition \ref{PropEst}.
Fix $j$ and the corresponding $A_j^{R,m}$. (1) in Lemma \ref{EstLemma} tells us that if $w,r\in A_j^{R,m}$ for which $(i_r,j_r)\in B((i_w,j_w),\frac{R_w}{\log^5 R_w})$ then \eqref{isom} holds. Indeed, if not that \eqref{largehor} holds, but then the horizontal distances at $r$ is much larger that the horizontal distance at $w$ hence they can not be in the same $A_j^{R,m}$ (horizontal distances in $A_j^{R,m}$ differ multiplicatively at most by $2$). Now \eqref{isom} means that the points at times $w,r$ move isometrically: their horizontal distance has to be the same (we have an isometry in the base). Then we use (2) to say that if \eqref{isom} holds, then the number of $r\in [-M,M]\cap \Z$ for which the points matching is less that $\frac{M}{\log^2M}$ (here $M$ is one of $R_w^{\frac{1}{1+|\gamma_2|}+\epsilon_0}$ or $R_w^{1-\epsilon_0}$). Hence the number of arrows is $M/\log^2M$ small. Now partitioning the interval $[0,R]$ into windows of size $M$, using the fact that on each we have at most $M/\log^2M$ good arrows and summing over the windows gives \textbf{(c)}.

\begin{proof}[Proof of Proposition \ref{PropEst}]
Fix $\delta>0$. We only have to prove $\textbf{(c)}$ using $(1)$ and $(2)$ of Lemma \ref{EstLemma}.

Fix $j$ as in $\textbf{(c)}$. Divide the interval $[0,R]$ into disjoint intervals $I_1,\ldots,I_k$ of length $\frac{(2^j)^{\frac{1}{1+|\gamma_2|}+\epsilon_0}}{\sqrt{\log2^j}}$ or $\frac{(2^j)^{1-\epsilon_0}}{\sqrt{\log2^j}}$ by the following procedure. Fix the smallest element $w_1\in W=\{r\in [0,R]:(i_r,j_r)\in U_R(x,y)\times U_R(x',y')\}\cap A^{R,m}_{j}((x,y),(x',y'))$. If $(i_{w_1}, j_{w_1})$ satisfies $(1)$ (satisfies $(2)$) let $I_1$ be an interval with right endpoint $w_1$ and length $l_j^1:=\frac{(2^j)^{\frac{1}{1+|\gamma_2|}+\epsilon_0}}{\sqrt{\log2^j}}$ (and length $l_j^2:=\frac{(2^j)^{1-\epsilon_0}}{\sqrt{\log2^j}}$). Now inductively for $u> 1$, we pick $w_u$ to be the smallest  element in $W\setminus(I_1\cup...\cup I_{u-1})$. According to whether $(i_{w_u},j_{w_u})$ satisfies $(1)$ or $(2)$, we let $I_u$ to be the interval with right endpoint $w_u$ and length $l_j^1$ or $l^2_j$. We continue until we cover $W$.
Notice that
Since $2^j\leq R^{\frac{1}{1-\epsilon_0}}$, it follows that $k>1$. Moreover by definition, we have
\begin{equation}\label{qwe}
|\{i\in \{1,\ldots,k\}: |I_i|=l_j^1\}|\leq  \left[\frac{R}{l_j^1}\right] +2
\end{equation}
and
\begin{equation}\label{weq}
|\{i\in \{1,\ldots,k\}: |I_i|=l_j^2\}|\leq  \left[\frac{R}{l_j^2}\right] +2.
\end{equation}
For each $i$ by definition of $A_j^{R,m}$, we have
\begin{equation}
R_{w_i}^{-1}=L_H(w_i)\in[2^{-j-1},2^{-j}].
\end{equation}
So
\begin{equation}
\frac{R_{w_i}}{\log^5R_{w_i}}\geq|I_i|.
\end{equation}

Therefore, by $(1)$ and the definition of $A_j^{R,m}$, we have
\begin{equation}
\begin{aligned}
\{r\in [0,R]:(i_r,j_r)\in U_R(x,y)&\times U_R(x',y')\}\cap A^{R,m}_{j}\left((x,y),(x',y')\right)\cap I_i\\
&\subset\{r\in I_i: \eqref{isom}\text{ holds },L(r)<2m^{-1}\}.
\end{aligned}
\end{equation}

By $(2)$ and the definition of $I_i$, we have one of the following:
\begin{equation}\label{eq1}
|\{r\in I_i:\eqref{isom}\text{ holds },L(r)<2m^{-1}\}|\leq\frac{(2^j)^{1-\epsilon_0}}{\log^22^j}.
\end{equation}
or
\begin{equation}\label{eq2}
|\{r\in I_i:\eqref{isom}\text{ holds },L(r)<2m^{-1}\}|\leq\frac{(2^j)^{\frac{1}{1+|\gamma_2|}+\epsilon_0}}
{\log^22^j}.
\end{equation}
Summing over $i\in\{1,\ldots,k\}$ by \eqref{eq1}, \eqref{qwe}, \eqref{eq2} and \eqref{weq}, we get
\begin{equation}\label{MeaEst}
\begin{aligned}
&|\{r\in [0,R]:(i_r,j_r)\in U_R(x,y)\times U_R(x',y')\}\cap A^{R,m}_{j}((x,y),(x',y'))|\\
&\leq\frac{R}{l_j^2}\frac{(2^j)^{1-\epsilon_0}}{\log^22^j}+
\frac{R}{l_j^1}\frac{(2^j)^{\frac{1}{1+|\gamma_2|}+\epsilon_0}}{\log^22^j}\\
&\leq\frac{R}{\log^{\frac{3}{2}}2^j}+\frac{R}{\log^{\frac{3}{2}}2^j},
\end{aligned}
\end{equation}
the last inequality by the definition of $l_j^1$ and $l_j^2$.
By \eqref{MeaEst}, we have
\begin{equation}
|\{r\in [0,R]:(i_r,j_r)\in U_R(x,y)\times U_R(x,y)\}\cap A^{R,m}_{j}((x,y),(x',y'))|\leq\frac{R}{j^{1.4}}.
\end{equation}
This finishes the proof.
\end{proof}

\section{Proof of Lemma \ref{EstLemma}}\label{proofestlemma}
Lemma \ref{EstLemma} is the crucial part in the proof of Theorem \ref{MainTheorem}. All propositions so far were based on some general combinatorial considerations and did not use too much of the flows we deal with. It is Lemma \ref{EstLemma} where specific properties of Kochergin flows play important role. Before we give a proof let us outline main ideas.

\paragraph{Sketch of the proof.} We first need to define the good set $B$. This is done by some standard ergodic theorem type of reasoning. We want points in $B$ to approach the singularity in a controlled way and so that we can use Lemmas \ref{koks} and \ref{ConDer} which allow to control relative speed between points. The set $U_R$ is the set of good times, i.e.  for times in $U_R$ we want to stay far away from the cusp and have good estimates for derivative. The core of the proof is (1) and (2). The idea behind (1) is the following: take two close points $x,x'\in \T^{f_1}$ (the same happens for $y,y'\in \T^{f_2}$). We look at their horizontal distance at the begining and after time $t<\frac{R_w}{\log^5 R_w}$. What is important is that the window is shorter (by a power of $\log$) than their horizontal distance. Now (1) tells us that either the points move isometrically together \eqref{isom} or, if not their horizontal distance at time $t$ is multiplicatively large compared to the starting distance. To get this we use diophantine assumptions -- the orbit of a point can not come too close if the time we iterate is to short.
For the proof of (2) we assume that points move isometrically. In this case for them to be close after time $t$ means that $f_1^{(n)}(x)-f_1^{(n)}(x')$ is close to $f_2^{(n)}(y)-f_2^{(n)}(y')$ (since first coordinates move isometrically, flow coordinates have to match). But the first birkhoff sums are of order $n^{1+\gamma_1}(x_h-x_h')$ and the second are  $n^{1+\gamma_2}(y_h-y_h')$. Moreover since by assumptions points move isometrically, $x_h-x'_h$ and $y_h-y'_h$ are constant. This amounts to the problem $|C_1n^{1+\gamma_1}-C_2n^{1+\gamma_2}|$ is small. But $\gamma_1\neq \gamma_2$ so this expression can not be small on the whole interval.

For $\delta>0$  let (see Lemma \ref{ConDer})
\begin{equation}\label{ef}
F=F_\delta:=\prod_{i=1,2}\left(S^i(n_1)\cap W^i\cap\{x\in\mathbb{T}^{f_i}:f_i(x_h)<
\delta^{-\frac{3}{1-\gamma_i}}\}\right)
\end{equation}
It follows by the definition of $S^i(n_1)\cap W^i$ that $\mu^{f_1}\times\mu^{f_2}(F)\geq 1-\delta^2$.




\paragraph{Set of good points.}\label{par}
By ergodic theorem (for $\chi_F$), we know that there exist a set $B=B_\delta\subset\mathbb{T}^{f_1}\times\mathbb{T}^{f_2}$, $\mu^{f_1}\times\mu^{f_2}(B)>1-\delta$ and there exists $n_3(\delta)\in\mathbb{N}$ such that for every $(x,y)\in B$ and $R\geq n_3(\delta)$, we have

\begin{equation}
|\{k\in[0,R]:(T^{f_1}\times T^{f_2})_{k}(x,y)\in F\}|\geq(1-\delta)R.
\end{equation}

For $(x,y)\in B$ let
\begin{equation}\label{ur}
\begin{aligned}
U_R=U_R(x,y):=\{k\in[0,R]:(T^{f_1}\times T^{f_2})_{k}(x,y)\in F\}.
\end{aligned}
\end{equation}
Notice that $\lambda(U_R)\geq\frac{99R}{100}$ ($\delta$ is small) and (see Section \ref{part})
\begin{equation}\label{ukm}
 \text{for every }k\in U_R, (T^{f_1}\times T^{f_2})_{k}(x,y)\in K^1_{\delta^{-1}}\times K^2_{\delta^{-1}}.
\end{equation}



Notice that with this definition of $B$, \textbf{(a)} and \textbf{(b)} follow automatically by \eqref{ur} and \eqref{ukm} (defining $m_\delta:=\delta^{-1}$).
Therefore we only need to prove (1) and (2).
\subsection{Proof of $(1)$}
$(1)$ is a straightforward consequence of the following lemma.
\begin{lm}Let $i=1,2$. There exists $c(\alpha_i,f_i)>0$ such that for $z,z'\in\T^{f_i}$ for which $W^{-1}:=d_H(z,z')<c(\alpha_i,f_i)$, we have for every $t\in[0,\frac{W}{\log^4W}]$ either $d_H(T^{f_{i}}_tz,T^{f_{i}}_tz')=d_H(z,z')$ or
$d_H(T^{f_{i}}_tz,T^{f_{i}}_tz')>100d_H(z,z')$.
\end{lm}
\begin{proof} Recall that since $\alpha_i\in \mathscr{D}$, we have
\begin{equation}\label{DioEst}
\inf_{|k|\leq m}\|k\alpha_i\|\geq\frac{C(\alpha_i)}{m\log^2m}.
\end{equation}

Fix $t\in [0,\frac{W}{\log^4W}]$. Then (by the definition of special flow) and $f_i>0$ we know that the first coordinates of $T^{f_{i}}_tz$ and $T^{f_{i}}_tz'$ are $z_h+m_t\a_i$ and $z'_h+n_t\a_i$ for some $0\leq m_t,n_t\leq\frac{W}{(\inf_\T f_i)\log^4W}\leq \frac{W}{\log^3W}$ provided that $W\geq C'(f_i)$ for some constant $C'(f_i)>0$. If $m_t=n_t$ then $d_H(T_t^{f_i}z,T_t^{f_i}z')=d_H(z,z')$ and the proof is finished . If $m_t\neq n_t$, then by \eqref{DioEst}
\begin{multline*}
d_H(T^{f_{i}}_tz,T^{f_{i}}_tz')=\|z_h-z'_h+(m_t-n_t)\a_i\|\geq
\|(m_t-n_t)\a_i\|- W^{-1}\geq\\
 \inf_{|k|\leq \frac{W}{\log^3W}}\|k\alpha_i\|-W^{-1}\stackrel{\eqref{DioEst}}{\geq} \frac{C(\alpha_i)\log W}{W}\geq 100W^{-1},
\end{multline*}
for $W\geq c'(\alpha_i)$. This finishes the proof.
\end{proof}
Now to get $(1)$, if \eqref{isom} does not hold (assume wlog that the first part of \eqref{isom} does not hold) we apply the above lemma for $t=i_r-i_w$ and $z=x_w, z'=x'_w$. This finishes the proof of $(1)$.





\subsection{Proof of $(2)$}
 We will use the notation from Definition \ref{eano}. To simplify the notation we will denote the horizontal (circle) coordinate of a point $z_w\in T^{f_i}$, $i=1,2$ by the same symbol, i.e. $(z_w)_h=z_w$. It will be clear from the context whether we consider $z_w$ as a point in the flow space or a point on the circle. Fix $(i_w,j_w)\in U_R\times U_R$. It follows that $x_w\in S^1(n_1)\cap W^1, y_w\in S^2(n_1)\cap W^2$ (see \eqref{ur} and \eqref{ef})

 We claim that for every $r$ such that \eqref{isom} holds and $L(r)<1/4$, we have
\begin{equation}\label{bird}
\begin{aligned}
|(f_1^{(N(x_w,i_r-i_w))}(x_w)&-f_1^{(N(x_w,i_r-i_w))}(x'_w))\\
&-(f_2^{(M(y_w,i_r-i_w))}(y_w)-f_2^{(M(y_w,i_r-i_w))}(y'_w))|\leq \frac{1}{2}
\end{aligned}
\end{equation}

Indeed $L(r)<1$ means that the second coordinates of $(x_r, x'_r)$ and $(y_r,y'_r)$ are close. However $x_r=T^{f_1}_{i_r-i_w}(x_w)$ and $x_r=T^{f_2}_{j_r-j_w}(x_w)$ (the same for $x'_r$ and $y'_r$). by \eqref{isom} we know that the action on the circle coordinate for $x_w$ and $x'_w$ is isometric, hence by the definition of special flow

$$
\begin{aligned}
&|[(i_r-i_w)-f^{(N(x_w,i_r-i_w))}(x_{w})]-
[(j_r-j_w)-f^{(N(x_w,i_r-i_w))}(x'_{w})]|<1/4,\\
&|[(i_r-i_w)-g^{(M(y_w,i_r-i_w))}(y_{w})]-
[(j_r-j_w)-g^{(M(y_w,i_r-i_w))}(y'_{w})]|<1/4.
\end{aligned}
$$
Then \eqref{bird} follows by triangle inequality.
Moreover, for every $r\in\N$ such that $(i_r,j_r)\in B((i_w,j_w),R_w^{1-\epsilon_0})$, \eqref{bird} is equivalent to
\begin{equation}\label{maineq}
|f_1'^{(N(x_w,i_r-i_w))}(\theta_r)(x_{w}-x'_{w})-
f_2'^{(M(y_w,i_r-i_w))}(\theta'_r)(y_w-y'_w)|\leq \frac{1}{2}.
\end{equation}
for some $\theta_r\in[x_w,x'_w]$ and $\theta'_r\in[y_w,y'_w]$. Indeed, this just follows by the fact $x_w,x'_w\in S_1(n_1)$, $y_w,y'_w\in S_2(n_1)$ (see \eqref{si}) and by Lemma \ref{new.l} so $f_1^{(N(x_w,i_r-i_w))}$ is differentiable on $[x_w,x'_w]$ and
$f_2^{(M(y_w,i_r-i_w))}$ is differentiable on $[y_w,y'_w]$.

For $T\in \R$ let  $G_T$ be as in Lemma \ref{ConDer}. We will assume that $(2)$ does not hold and get a contradiction with \eqref{maineq}. Then we have the following crucial

\textbf{Claim}. There exists $r_1>r_0>w$, such that, : \begin{enumerate}
    \item[$(i)$] $(i_{r_0},j_{r_0})\in B\left((i_w,j_w),R_w^{\frac{1}{1+|\gamma_2|}+\epsilon_0}\right)\setminus B\left((i_w,j_w),\frac{1}{2}R_w^{\frac{1}{1+|\gamma_2|}+\epsilon_0}\log^{-2}R_w\right)$;
	 \item[$(ii)$] $i_{r_0}-i_w\in G_{R_w^{\frac{1}{1+|\gamma_2|}+\epsilon_0}}$;
 \item[$(iii)$] $(i_{r_1},j_{r_1})\in B\left((i_w,j_w),R_w^{1-\epsilon_0}\right)\setminus B\left((i_w,j_w),\frac{1}{2}R_w^{1-\epsilon_0}\log^{-2}R_w\right)$;
	 \item[$(iv)$] $i_{r_1}-i_w\in G_{R_w^{1-\epsilon_0}}$;
	\end{enumerate}
Before we prove the \textbf{Claim} let us show how it gives a contradiction and hence also proofs $(2)$. Recall that $$R_w^{-1}=L_H(w)=\max(d_H(x_w,x'_w),d_H(y_w,y'_w)).$$ If $R_w^{-1}=d_H(x_w,x'_w)$ then using Lemma \ref{ConDer} for $x_w\in S^1(n_1)\cap W^1$ and $y_w\in S^2(n_1)\cap W^2$ and $T=R^{1-\epsilon_0}$, using $(iii)$ and $(iv)$ we get (since $\gamma_1<\gamma_2<0$)
\begin{multline*}
|f_1'^{(N(x_w,i_{r_1}-i_w))}(\theta_r)(x_{w}-x'_{w})|\geq R_w^{(1-2\epsilon_0)(1+|\gamma_1|-\epsilon_2)}\|x_w-x'_w\|\geq\\
 R_w^{1+|\gamma_2|+\epsilon_2}\|y_w-y'_w\|+100\geq
|f_2'^{(M(y_w,i_r-i_w))}(\theta'_r)(y_w-y'_w)|+100
\end{multline*}
which contradicts \eqref{maineq}.

Hence we have to consider the case $R_w^{-1}=d_H(y_w,y'_w)$. In this case we will use $(i)$-$(iv)$ of the \textbf{Claim}.

By \eqref{maineq} for $r_0$ and $r_1$, we have
\begin{equation}\label{est}
\begin{aligned}
&\frac{|f_2'^{(M(y_w,i_{r_0}-i_w))}(\theta'_{r_0})|\|y_w-y'_w\|-\frac{1}{2}}
{|f_1'^{(N(x_w,i_{r_0}-i_w))}(\theta_{r_0})|}\leq\|x_w-x'_w\|\\
&\leq\frac{|f_2'^{(M(y_w,i_{r_1}-i_w))}(\theta'_{r_1})|\|y_w-y'_w\|+\frac{1}{2}}
{|f_1'^{(N(x_w,i_{r_1}-i_w))}(\theta_{r_1})|}.
\end{aligned}
\end{equation}
We use Lemma \ref{ConDer} twice, i.e. first for $x_w,y_w$ and $i_{r_0}-i_w$ using $(i)$ and $(ii)$ and then for $x_w,y_w$ and $i_{r_1}-i_w$ using $(iii)$ and $(iv)$ to get (using that $\|y_w-y'_w\|=R_w^{-1}$ and \eqref{est})
$$
\frac{R_w^{(\frac{1}{1+|\gamma_2|}+\frac{1}{2}\epsilon_0)
(1+|\gamma_2|-\epsilon_2)-1}-\frac{1}{2}}
{R_w^{(\frac{1}{1+|\gamma_2|}+\epsilon_0)(1+|\gamma_1|+\epsilon_2)}}\leq
\frac{R_w^{(1-\epsilon_0)(1+|\gamma_2|+\epsilon_2)-1}+\frac{1}{2}}
{R_w^{(1-2\epsilon_0)(1+|\gamma_1|-\epsilon_2)}},
$$
this however is a contradiction with the choice of $\epsilon_0,\epsilon_2>0$ in Definition \ref{eps0}. Therefore we only have to prove the \textbf{Claim}.

\paragraph{Proof of the \textbf{Claim}.}
We will  prove $(i)$ and $(ii)$ the proof of (iii) and $(iv)$ follows the same lines. Since we assume that $(2)$ does not hold, in particular if follows that \eqref{sh} is not satisfied. For simplicity denote $Z_w:=R_w^{\frac{1}{1+|\gamma_2|}+\epsilon_0}$.
 Notice that by Lemma \ref{ConDer}, the measure of $G_{Z_w}\subset [0,Z_w]$ is at least $Z_w(1-4\log^{-3}Z_w)$. Moreover, for $C=\min_{i=1,2}(\inf_\T f_i)$. Then
\begin{equation}\label{int.point}
\left|\{n\in[0,Z_w]\cap \Z: n\in G_{Z_w}\}\right|\geq Z_w(1-4C^{-1}\log^{-3}Z_w),
\end{equation}
this follows by the fact that $N(x,t)$ is locally constant (on intervals of length $\geq C$).
Since \eqref{sh} does not hold, we have for
$$
\mathcal{B}_w:=B\left((i_w,j_w),Z_w\right)\setminus B\left((i_w,j_w),\frac{1}{2}Z_w\log^{-2}R_w\right)
$$
that
$$
\lambda(\mathcal{B}_w)\geq \frac{1}{2}Z_w\log^{-2}R_w
\gg 4C^{-1}Z_w\log^{-3}Z_w.
$$
Therefore and by \eqref{int.point}, we get
$$
\mathcal{B}_w\cap \left\{(i_r-i_w,j_r-j_w): r\in G_{Z_w}\right\}\neq \emptyset.
$$
Take $(i_{r_0},j_{r_0})$ to be any point in the intersection. This finishes the proof.

\section{Appendix}
\subsection{Proof of Lemma \ref{new.l}}
\begin{proof}
We will give the proof for $i=1$ (the proof in case $i=2$ follows the same lines).
Notice that since $z\in S^1(n_1)$ and $N(z,t)\leq\frac{t}{c}$, for $t<\frac{d_H(z,z')^{-1}}{|\log^7d_H(z,z')|}$, we have
\begin{equation}
\begin{aligned}
d_H(T^{f_1}_tz,0)\geq\frac{1}{\frac{d_H(z,z')^{-1}}{|\log^7d_H(z,z')|}\log^3(\frac{d_H(z,z')^{-1}}{|\log^7d_H(z,z')|})}\geq d_H(z,z')|\log^2d_H(z,z')|
\end{aligned}
\end{equation}
Therefore, we have that for $t\leq\frac{d_H(z,z')^{-1}}{|\log^7d_H(z,z')|}$
\begin{equation}
0\notin[z_h+N(z,t)\alpha_1,z'_h+N(z,t)\alpha_1].
\end{equation}

Therefore, to finish the proof, it is enough to show that for $z\in S^1(n_1)$,
\begin{equation}\label{Nestimate}
N(z,t)\geq\frac{t}{\log^5t}.
\end{equation}

Suppose that \eqref{Nestimate} is not true. Notice that by definition of $N(z,t)$, we have
\begin{equation}
t<f_1^{(N(z,t)+1)}(z).
\end{equation}

Therefore and by \eqref{DK1}, we have,
\begin{equation}
\begin{aligned}
t<f_1^{(N(z,t)+1)}(z)&\leq N(z,t)\log^3N(z,t)<\frac{t}{\log^5t}\log^3(\frac{t}{\log^5t})\\
&\leq \frac{t}{\log^2t}.
\end{aligned}
\end{equation}
This contradiction shows that \eqref{Nestimate} holds.

\end{proof}
\subsection{Proof of Lemma \ref{ConDer}}

We will conduct the proof for $S^1(n_1)\cap W^1$. by the definition of $S^1(n_1)$ and $N(x,t)\leq\frac{t}{c}$($c=\inf_{\mathbb{T}}f_i$), we have,
\begin{equation}\label{disEsit1}
\min_{j\in[0,t)}\|x_h+j\alpha_1-0\|\geq\frac{1}{t\log^4t}.
\end{equation}

Notice that $d_H(\theta_h,x_h)\leq\frac{1}{T\log^{2P_1}T}$ and $t\in[0,T]$, thus we have,
\begin{equation}\label{disEsit2}
\min_{j\in[0,t)}\|\theta_h+j\alpha_1-0\|\geq\frac{1}{t\log^{5}t}.
\end{equation}

By \eqref{DK2} for $\theta_h$ and $s\in N$ such that $q_s\leq N(x,t)<q_{s+1}$, we have
\begin{equation}
\begin{aligned}
|f_1'^{(N(x,t))}(\theta_h)|&\leq f_1'((\theta_h)_{\min}^{N(x,t)})+|\gamma|8N(x,t)^{1+|\gamma_1|+\frac{1}{2}\epsilon_2}\\
&\leq (N(x,t)\log^5N(x,t))^{1+|\gamma_1|+\frac{1}{2}\epsilon_2}\\
&\leq N(x,t)^{1+|\gamma_1|+\frac{2}{3}\epsilon_2}<t^{1+|\gamma_1|+\epsilon_2},
\end{aligned}
\end{equation}
which proves the upper bound.

To get the lower bound, notice that for $x\in S^1(n_1)\cap W^1\subset W^1$, outside a set of $t$'s of measure at most $\frac{1}{\log^3T}$, by Proposition \ref{SetPro}, we have
\begin{equation}
|f_1'^{(N(x,t))}(x_h)|\geq\frac{N(x,t)^{1+|\gamma_1|}}{\log^{P_1}N(x,t)}.
\end{equation}

Fix a ``good'' $t$ as above. Notice by Lemma \ref{new.l} (for $\eta\in[\theta_h,x_h]$),
\begin{equation}
\begin{aligned}
|f_1'^{(N(x,t))}(\theta_h)-f_1'^{(N(x,t))}(x_h)|=|f_1''^{(N(x,t))}(\eta)|\|\theta_h-x_h\|.
\end{aligned}
\end{equation}

Moreover, by \eqref{disEsit1} and \eqref{disEsit2}, we have
\begin{equation}
\min_{j\in[0,t)}\|\eta+j\alpha_1-0\|\geq\frac{1}{t\log^5t}.
\end{equation}

Thus, by \eqref{DK3}, we have
\begin{equation}
|f_1''^{(N(x,t))}(\eta)|\leq N(x,t)^{2+|\gamma_1|}\log^{15}N(x,t)
\end{equation}
and so(since $N(x,t)<ct<cT$)
\begin{equation}
\begin{aligned}
|f_1''^{(N(x,t))}(\eta)|\|\theta_h-x_h\|&\leq (N(x,t)^{2+|\gamma_1|}\log^{15}N(x,t))\frac{1}{T\log^{2P_1}T}\\
&\leq \frac{1}{2}\frac{N(x,t)^{1+|\gamma_1|}}{\log^{P_1}N(x,t)}\leq\frac{1}{2}|f_1'^{(N(x,t))}(x_h)|.
\end{aligned}
\end{equation}

Finally, we have
\begin{equation}
\begin{aligned}
|f_1'^{(N(x,t))}(\theta_h)|&\geq|f_1'^{(N(x,t))}(x_h)|-|f_1'^{(N(x,t))}(\theta_h)-f_1'^{(N(x,t))}(x_h)|\\
&\geq\frac{1}{2}|f_1'^{(N(x,t))}(x_h)|\geq\frac{1}{2}\frac{N(x,t)^{1+|\gamma_1|}}{\log^{P_1}N(x,t)}\\
&\geq N(x,t)^{1+|\gamma_1|-\frac{2}{3}\epsilon_2}>t^{1+|\gamma_1|-\epsilon_2}.
\end{aligned}
\end{equation}
where the last inequality follows by \eqref{Nestimate} (since $x\in S^1(n_1)$).

This finishes the proof of Lemma.

\subsection*{Acknowledgements}
The authors would like to thank Anatole Katok for his patience, help, encouragement and deep insight. The authors would also like to thank Jean-Paul Thouvenot for several discussions on the subject. The authors would also like to thank the anonymous referee's many helpful suggestions which improve the article's quality a lot.


\begin{thebibliography}{HD}




\normalsize
\baselineskip=17pt


\bibitem[B]{Benhenda}  M. Benhenda, {\it An uncountable family of pairwise non-Kakutani equivalent smooth diffeomorphisms} J. Anal. Math. 127 (2015), 129–178.
\bibitem[FFK]{FFK} B. Fayad, G. Forni, A. Kanigowski, {\it Lebesgue spectrum for area preserving flows on the two torus}, submitted, arXiv:1609.03757.
\bibitem[FK]{KanigowshiFayad} B. Fayad, A. Kanigowski, {\it Multiple mixing for a class of conservative surface flows}, Invent. Math. 203.2 (2016): 555-614.
\bibitem[F]{Feldman} J. Feldman, {\it New $K$-automorphisms and a problem of Kakutani}, Israel J. Math. 24.1 (1976): 16-38.
\bibitem[Ka]{Kanigowski} A. Kanigowski, {\it Slow entropy for some smooth flows on surfaces}, submitted, arXiv:1612.09364.
\bibitem[K1]{Katok1} A. Katok, {\it Time change, monotone equivalence, and standard dynamical systems}, English translation: Soviet Math. Dokl. 16 (1975), no. 4, 986–990 (1976).
\bibitem[K2]{Katok2} A. B. Katok, {\it Monotone equivalence in ergodic theory}, Mathematics of the USSR-Izvestiya 11.1 (1977): 99.
\bibitem[Ki]{Kinchin} A. Ya. Khinchin, {\it Continued fractions}, University of Chicago Press, (1964).
\bibitem[Ko]{Koch} A. V. Kochergin, {\it Mixing in special flows over a shifting of segments and in smooth flows on surfaces},
Mat. Sb., 96 138 (1975): 471-502.
\bibitem[ORW]{OrnRudWei}D. S. Ornstein, D. Rudolph and B. Weiss {\it Equivalence of measure preserving transformations}. Vol. 262. American Mathematical Soc., (1982).
\bibitem[R1]{Ratner} M. Ratner, {\it Some invariants of Kakutani equivalence}, Israel J. Math. 38.3 (1981): 231-240.
\bibitem[R2]{Ratner1} M. Ratner, {\it The Cartesian square of the horocycle flow is not loosely Bernoulli}, Israel J. Math. 34.1 (1979): 72-96.
\bibitem[R3]{Rat3} M. Ratner, {\it Horocycle flows are loosely Bernoulli}, Israel J. Math. (1978), 31: 122-132.
\bibitem[T]{Tho} J-P. Thouvenot, {\it Entropy, isomorphism and equivalence in ergodic theory}, Handbook of dynamical systems 1 (2002): 205-238.

\end{thebibliography}
\end{document}